\documentclass{article}
\usepackage[utf8]{inputenc}
\usepackage[]{graphicx}
\usepackage{amsthm,amsmath,amssymb,tikz,tikz-cd,amsmath,mathrsfs,mathtools,multicol,dirtytalk,tabularx,xy,lipsum,url,enumitem,cmll}

\usepackage{enumitem}

\usepackage{blkarray}
\usepackage{tabularray}
\usepackage[colorlinks=true,linkcolor=blue,citecolor=blue]{hyperref}

\numberwithin{equation}{section}

\theoremstyle{plain}
\newtheorem{theorem}{Theorem}[section]
\newtheorem{lemma}[theorem]{Lemma}
\newtheorem{proposition}[theorem]{Proposition}

\theoremstyle{definition}
\newtheorem{definition}[theorem]{Definition}
\newtheorem{remark}[theorem]{Remark}

\title{\textbf{Cylindric quasi-implication algebras}}
\author{Joseph McDonald\footnote{University of Alberta, Department of Philosophy, Edmonton, T6G 2E7, Canada.}\footnote{Email: \url{jsmcdon1@ualberta.ca}, ORCID: \url{https://orcid.org/0000-0003-3824-1988}}.}
\date{}

\begin{document}

\maketitle
\par
\vspace{-.6cm}

\begin{abstract}
    In this note, we study the operation of Sasaki hook within the setting of quantum cylindric algebras by introducing cylindric quasi-implication algebras. It is first demonstrated that every quantum cylindric algebra can be converted into a cylindric quasi-implication algebra and conversely that every cylindric quasi-implication algebra gives rise to a quantum cylindric algebra. These constructions are then shown to induce an isomorphism between the category $\mathbf{CQIA}$ of cylindric quasi-implication algebras and the category $\mathbf{QCA}$ of quantum cylindric algebras. We then give two alternative constructions of a cylindric orthoframe $X_A$ from a cylindric quasi-implication algebra $A$. The first construction of $X_A$ arises via the non-zero elements of $A$ and generalizes the construction given by Harding in the setting of cylindric ortholattices from the perspective of MacNeille completions. The second construction of $X_A$ arises via the proper filters of $A$ and generalizes the construction given by McDonald in the setting of cylindric ortholattices from the perspective of canonical completions.   

    \par
\vspace{.2cm}
\end{abstract}
\section{Introduction}
Abbott \cite{abbott671,abbott672} introduced implication algebras as algebraic models for the operation of (material) implication in a Boolean algebra which is defined by the Boolean algebra polynomial $\pi_c(x,y)=x^{\perp}\vee y$. An implication algebra consists of a set $A$ equipped with a binary operation $\cdot\colon A\times A\to A$ satisfying the contraction, quasi-commutativity, and exchange equations: 
\[(x\cdot y)\cdot x=x,\hspace{.2cm}(x\cdot y)\cdot y=(y\cdot x)\cdot x,\hspace{.2cm}x\cdot (y\cdot z)=y\cdot(x\cdot z).\]
Abbot showed that $\langle B;\pi_c\rangle$ is an implication algebra for every Boolean algebra $B$ and that every bounded implication algebra gives rise to a Boolean algebra.

An \emph{ortholattice} is a bounded lattice $\langle A;\wedge,\vee,0,1\rangle$ equipped with an additional operation $^{\perp}\colon A\to A$, known as an \emph{orthocomplementation}, which is an order-inverting involutive complementation. An ortholattice is a \emph{modular ortholattice} if it satisfies the quasi-inequation $x\leq y\Rightarrow(x\vee y)\wedge z\leq x\vee(y\wedge z)$. Moreover, an ortholattice is an \emph{orthomodular lattice} if it satisfies the quasi-equation $x\leq y\Rightarrow y=x\vee(x^{\perp}\wedge y)$. Ortholattices form non-distributive generalizations of Boolean algebras and were popularized by Birkhoff and von Neumann \cite{birkhoff} who observed that the closed linear subspaces of a finite-dimensional Hilbert space form a modular ortholattice. Soon after, it was realized that the closed subspaces of an infinite-dimensional Hilbert space form an orthomodular lattice and that orthomodular lattices are suitable candidates as an algebraic model for quantum logic in the same way in which Boolean algebras form an algebraic model for the classical propositional calculus.

Hardegree \cite{hardegree1} proposed the following as the minimal conditions that an implication operation $\cdot\colon A\times A\to A$ should be required to satisfy in an ortholattice:
\begin{enumerate}
    \item $x\leq y\Rightarrow x\cdot y=1$ (Law of Implication); 
    \item $x\wedge(x\cdot y)\leq y$ (Modus Ponens); 
    \item $y^{\perp}\wedge(x\cdot y)\leq x^{\perp}$ (Modus Tollens). 
\end{enumerate}
 Hardegree \cite{hardegree3} showed that there are exactly 3 ortholattice polynomials satisfying these conditions in an orthomodular lattice: 
\begin{enumerate}
    \item $\pi_s(x,y)=x^{\perp}\vee(x\wedge y)$ (Sasaki hook);
    \item $\pi_d(x,y)=(x^{\perp}\wedge y^{\perp})\vee y$ (Dishkant implication);
    \item $\pi_k(x,y)=(x\wedge y)\vee(x^{\perp}\wedge y)\vee(x^{\perp}\wedge y^{\perp})$ (Kalmbach implication).
\end{enumerate}
 One can easily demonstrate that when an ortholattice is distributive (i.e., a Boolean algebra), then: 
\[\pi_c(x,y)=\pi_s(x,y)=\pi_d(x,y)=\pi_k(x,y).\] Hardegree \cite{hardegree1, hardegree2} introduced quasi-implication algebras as algebraic models of the implicational fragment of orthomodular lattices based on the operation of Sasaki hook and showed that bounded quasi-implication algebras stand in bijective correspondence with orthomodular lattices.  Recently, McDonald \cite{mcdonald1} studied quasi-implication algebras within the setting of quantum monadic algebras by introducing monadic quasi-implication algebras. A quantum monadic algebra consists of an orthomodular lattice $A$ equipped with a closure operator $\exists\colon A\to A$, known as a \emph{quantifier}, whose closed elements form an orthomodular sub-lattice. Janowitz \cite{janowitz} introduced quantifiers on orthomodular lattices and Harding \cite{harding} studied them, as well as quantum cylindric algebras, for their connection to subfactor theory. The  closely related monadic ortholattices have been studied by Harding, McDonald, and Peinado \cite{harding3} from the perspective of MacNeille completions, canonical completions, and duality theory as well as by Lin and McDonald \cite{lin} from the perspective of functional representations.   

In this note, we extend the results and constructions obtained in \cite{mcdonald1} to the setting of quantum cylindric algebras by introducing cylindric quasi-implication algebras. A \emph{quantum cylindric algebra} consists of an orthomodular lattice $A$ equipped with a family $(\exists_i)_{i\in I}\colon A\to A$ of pairwise commuting quantifiers and a family $(\delta_{i,k})_{i,k\in I}$ of constants satisfying certain conditions. It is first demonstrated that every quantum cylindric algebra can be converted into a cylindric quasi-implication algebra and conversely that every cylindric quasi-implication algebra gives rise to a quantum cylindric algebra. These constructions are then shown to induce an isomorphism between the category $\mathbf{CQIA}$ of cylindric quasi-implication algebras and the category $\mathbf{QCA}$ of quantum cylindric algebras. We then give two alternative constructions of a cylindric orthoframe $X_A$ from a cylindric quasi-implication algebra $A$. The first construction of $X_A$ arises via the non-zero elements of $A$ and generalizes the construction given by Harding \cite{harding} (which extends the findings of MacLaren \cite{maclaren}) in the setting of cylindric ortholattices from the perspective of MacNeille completions. The second construction of $X_A$ arises via the proper filters of $A$ and generalizes the construction given by McDonald \cite{mcdonald2} (which extends the findings of Goldblatt \cite{goldblatt} as well as Harding, McDonald, and Peinado \cite{harding3}) in the setting of cylindric ortholattices from the perspective of canonical completions.

\section{Quantum cylindric algebras}
We first provide some preliminaries for quantum cylindric algebras. For more details consult \cite{harding, h}. 

\begin{definition}\label{ortholattice}
    An \emph{ortholattice} is an algebra $\langle A;\wedge,\vee,^{\perp},0,1\rangle$ of similarity type $\langle 2,2,1,0,0\rangle$ satisfying the following conditions: 
    \begin{enumerate}
        \item $\langle A;\wedge,\vee,0,1\rangle$ is a bounded lattice;
        \item the operation $^{\perp}\colon A\to A$ is an \emph{orthocomplementation}: 
        \begin{enumerate}
             \item $x\wedge x^{\perp}=0$; $x\vee x^{\perp}=1$;
            \item $x\leq y\Rightarrow y^{\perp}\leq x^{\perp}$;
            \item $x=x^{\perp\perp}$.
        \end{enumerate}
    \end{enumerate}
\end{definition}

Conditions $2(a)-2(c)$ imply that $^{\perp}\colon A\to A$ is an order-inverting involutive complementation. Ortholattices, unlike Boolean algebras, are not necessarily distributive. Indeed, an ortholattice $A$ is a Boolean algebra if and only if $A$ is distributive.    
\begin{definition}\label{oml}
    An \emph{orthomodular lattice} is an ortholattice $\langle A;\wedge,\vee,^{\perp},0,1\rangle$ satisfying the quasi-equation $x\leq y\Longrightarrow y=x\vee(x^{\perp}\wedge y)$. 
\end{definition}

The following provides a useful characterization of orthomodular lattices. 
\begin{proposition}[\cite{kalmbach}]\label{oml characterization}
For any ortholattice $A$, the following are equivalent: 
\begin{enumerate}
    \item $A$ is an orthomodular lattice (in the sense of Definition \ref{oml});
    \item $A$ satisfies  $x\vee y=x\vee(x^{\perp}\wedge(x\vee y))$ or $x\wedge y=x\wedge(x^{\perp}\vee(x\wedge y))$;
\end{enumerate}
\end{proposition}
\begin{remark}
    Orthomodular lattices form a variety in the sense of universal algebra and hence are closed under the formation of homomorphic images, subalgebras, and direct products. 
\end{remark}
\begin{definition}\label{monadic ortholattice}
A \emph{monadic ortholattice} is an algebra $\langle A;\wedge,\vee,^{\perp},0,1,\exists\rangle$ of type $\langle 2,2,1,0,0,1\rangle$ satisfying the following conditions:
\begin{enumerate}
    \item $\langle A;\wedge,\vee,^{\perp},0,1\rangle$ is an ortholattice;
    \item $\exists\colon A\to A$ is a unary operator, known as a \emph{quantifier}, satisfying:
        \begin{enumerate}
        \item $\exists 0=0$
        \item $x\leq \exists x$
        \item $\exists(x\vee y)=\exists x\vee\exists y$
        \item $\exists\exists x=\exists x$
        \item $\exists(\exists x)^{\perp}=(\exists x)^{\perp}$
    \end{enumerate}
\end{enumerate}
We call $\langle A;\wedge,\vee,^{\perp},0,1,\exists\rangle$ a \emph{quantum monadic algebra} whenever $\langle A;\wedge,\vee,^{\perp},0,1\rangle$ is an orthomodular lattice.
\end{definition}
Noting that $\exists x=x$ iff $\exists(x^{\perp})=x^{\perp}$, conditions 2(a) through 2(e) amount to asserting that a quantifier on an orthomodular lattice is a closure operator whose closed elements form an orthomodular sub-lattice.
\begin{definition}\label{quantum cylindric algebra}
    An \emph{I-dimensional cylindric ortholattice} is an algebraic structure $\langle A;\wedge,\vee,^{\perp},0,1,(\exists_i)_{i\in I},(\delta_{i,k})_{i,k\in I}\rangle$ satisfying the following conditions: 
    \begin{enumerate}
    \item $\langle A;\wedge,\vee,^{\perp},0,1\rangle$ is an ortholattice;
        \item $\exists_i\colon A\to A$ is a quantifier for each $i\in I$; 
        \item $\exists_i\exists_ka=\exists_k\exists_ia$ for all $i,k\in I$; 
        \item $(\delta_{i,k})_{i,k\in I}$ is a family of constants, known as the \emph{diagonal elements}, satisfying the following conditions: 
            \begin{enumerate}
        \item $\delta_{i,k}=\delta_{k,i}$ and $\delta_{i,i}=1$;
        \item $i,l\not=k\Rightarrow\exists_{k}(\delta_{i,k}\wedge \delta_{k,l})=\delta_{i,l}$.
    \end{enumerate}
    \end{enumerate}
    We call $\langle A;\wedge,\vee,^{\perp},0,1,(\exists_i)_{i\in I},(\delta_{i,k})_{i,k\in I}\rangle$ an \emph{I-dimensional quantum cylindric algebra} whenever the reduct $\langle A;\wedge,\vee,^{\perp},0,1\rangle$ is an orthomodular lattice. We call the reduct $\langle A;\wedge,\vee,^{\perp},0,1,(\exists_{i})_{i\in I}\rangle$ the \emph{I-dimensional diagonal-free quantum cylindric algebra reduct} of $A$ provided condition 1-3 are satisfied.   
\end{definition}

\section{Cylindric quasi-implication algebras}

In this section, we introduce cylindric quasi-implication algebras and construct the promised isomorphism between the category $\mathbf{QCA}$ of quantum cylindric algebras and the category $\mathbf{CQIA}$ of cylindric quasi-implication algebras. 

\begin{definition}\label{quasi-implication algebra}
    A \emph{quasi-implication algebra} is a magma $\langle A;\cdot\rangle$ satisfying: 
    \begin{enumerate}
        \item $(x\cdot y)\cdot x=x$;
        \item $(x\cdot y)\cdot(x\cdot z)=(y\cdot x)\cdot(y\cdot z)$;
        \item $((x\cdot y)\cdot(y\cdot x))\cdot x=((y\cdot x)\cdot(x\cdot y))\cdot y$.
    \end{enumerate}
\end{definition}
  
\begin{lemma}[\cite{hardegree1}]\label{lemma}
    Any quasi-implication algebra $\langle A;\cdot\rangle$ satisfies:
    \begin{enumerate}
        \item $x\cdot(x\cdot y)=x\cdot y$;
        \item $x\cdot x=(x\cdot y)\cdot(x\cdot y)$;
        \item $x\cdot x=y\cdot y$; 
        \item $(x\cdot y)\cdot(x\cdot z)=x\cdot((x\cdot y)\cdot z)$. 
    \end{enumerate}
    \end{lemma}
\begin{definition}\label{top element}
    Let $A$ be a quasi-implication algebra. Then define a constant $1\in A$ by $1:=x\cdot x$. 
\end{definition}
Note that by virtue of condition 3 of Lemma \ref{lemma}, the constant element $1$ is well-defined in any quasi-implication algebra. 
\begin{proposition}[\cite{hardegree1}]\label{top element lemma}
    Every quasi-implication algebra $A$ satisfies $1\cdot x=x$ and $x\cdot 1=1$ for all $x\in A$. 
\end{proposition}
\begin{theorem}[\cite{hardegree1}]\label{oml to qia}
    If $A$ is an orthomodular lattice, then $\langle A;\pi_s\rangle$ is a quasi-implication algebra. 
\end{theorem}
\begin{lemma}[\cite{hardegree1}]\label{partial order}
If $A$ is a quasi-implication algebra, then $\preceq\subseteq A^2$ defined by $x\preceq y$ iff $x\cdot y=1$ is a partial order.      
\end{lemma}
Notice that for any quasi-implication algebra $A$, we have $x\cdot 1=1$ for all $x\in A$ by Proposition \ref{top element lemma}, and hence $x\preceq 1$ for all $x\in A$ by Lemma \ref{partial order}. Hence the constant element $1\in A$ is the greatest element in the poset $\langle A;\preceq\rangle$. 

\begin{proposition}[\cite{hardegree1}]\label{left monotone}
    In any quasi-implication algebra $A$, if $y\cdot z=1$, then $(x\cdot y)\cdot(x\cdot z)=1$. 
\end{proposition}

\begin{definition}\label{bottom element}
    A \emph{bounded quasi-implication algebra} is a quasi-implication algebra $A$ equipped with a constant $0\in A$ defined by $0\cdot x=1$ for all $x\in A$. 
\end{definition}
By Lemma \ref{partial order} and Definition \ref{bottom element},  $0$ is the least element in the poset $\langle A;\preceq\rangle$ for any bounded quasi-implication algebra $A$. Moreover, since for any orthomodular lattice $A$ and $x\in A$, we have: \[\pi_s(0,x)=0^{\perp}\vee(0\wedge x)=1\vee(0\wedge x)=1\vee 0=1,\]  it follows that $\langle A;p_s\rangle$ further induces a bounded quasi-implication algebra.

Conversely to that of Lemma \ref{oml to qia}, the following result shows that every bounded quasi-implication algebra induces an orthomodular lattice. 

\begin{lemma}[\cite{hardegree1}]\label{quasi-implication algebra is an orthomodular lattice}
    Every bounded quasi-implication algebra forms an orthomodular lattice in which for all $x,y\in A$:
    \begin{enumerate}
        \item the orthocomplementation of $x$ is given by $x\cdot 0$;
        \item the least upper bound of $\{x,y\}$ is given by $((x\cdot y)\cdot(y\cdot x))\cdot x$;
        \item the greatest lower bound of $\{x,y\}$ is given by $((((x\cdot0)\cdot (y\cdot0))\cdot((y\cdot0)\cdot (x\cdot0)))\cdot (x\cdot0)\cdot0)$. 
    \end{enumerate}
\end{lemma}

The proceeding lemma will be exploited within the constructions and proofs throughout the remainder of this paper. 
\begin{lemma}[\cite{hardegree1}]\label{meets}
    If $A$ is a bounded quasi-implication algebra, then: \[((((x\cdot0)\cdot (y\cdot0))\cdot((y\cdot0)\cdot (x\cdot0)))\cdot (x\cdot0)\cdot0)=((x\cdot y)\cdot(x\cdot 0))\cdot 0\] so that $((x\cdot y)\cdot(x\cdot 0))\cdot 0$ is the greatest lower bound of $\{x,y\}$ in $A$.  
\end{lemma}

The following were introduced by McDonald \cite{mcdonald1} within the setting of quantum monadic algebras. 
\begin{definition}\label{def of mqia}
    A \emph{monadic quasi-implication algebra} is an algebra $\langle A;\cdot,0,\Diamond\rangle$ of similarity type $\langle 2,0,1\rangle$ satisfying the following conditions: 
    \begin{enumerate}
        \item $\langle A;\cdot,0\rangle$ is a bounded quasi-implication algebra; 
        \item $\Diamond\colon A\to A$ is a unary operation satisfying the following conditions: 
        \begin{enumerate}
            \item $\Diamond\Diamond x\cdot\Diamond x=1$ and $x\cdot\Diamond x=1$;
            \item $\Diamond(\Diamond x\cdot 0)=\Diamond x\cdot 0$ and $\Diamond 0=0$; 
            \item $\Diamond(((x\cdot 0)\cdot(y\cdot0))\cdot x)=((\Diamond x\cdot 0)\cdot(\Diamond y\cdot0))\cdot \Diamond x$
            
        \end{enumerate}
    \end{enumerate}
\end{definition}
Below is an immediate consequence of the axioms of monadic quasi-implication algebras and the induced partial order relation described in Lemma \ref{partial order}. 
\begin{proposition}[\cite{mcdonald1}]\label{idempotent}
    If $A$ is a monadic quasi-implication algebra, then $\Diamond\colon A\to A$ is idempotent and monotone.  
\end{proposition}
We now introduce the primary objects of study in this work; namely, the cylindric quasi-implication algebras. 
\begin{definition}\label{I-dimensional cylindric quasi-implication algebra}
    An \emph{I-dimensional cylindric quasi-implication algebra} is an algebraic structure $\langle A;\cdot,0,(\Diamond_i)_{i\in I}, (d_{i,k})_{i,k\in I}\rangle$ satisfying the following conditions: 
    \begin{enumerate}
        \item $\langle A;\cdot,0,\Diamond_i\rangle$ is a monadic quasi-implication algebra for all $i\in I$; 
        \item $\Diamond_i\Diamond_kx=\Diamond_k\Diamond_ix$ for all $i,k\in I$;
        \item $d_{i,k}=d_{k,i}$ and $d_{i,i}=1$ for all $i,k\in I$; 
        \item $i,l\not=k\Rightarrow\Diamond_k(((d_{i,k}\cdot d_{k,l})\cdot(d_{i,k}\cdot 0))\cdot 0)=d_{i,l}$. 
    \end{enumerate}
\end{definition}
The following lemma will be used to simplify the proof of Theorem \ref{alg to imp}.  
\begin{lemma}\label{useful lemma}
    If $A$ is an orthomodular lattice, then $\pi_s(\pi_s(x,y),x^{\perp})^{\perp}=x\wedge y$. 
\end{lemma}
\begin{proof}
    The calculation is straightforward and runs as follows: 
    \begin{align*}
        \pi_s(\pi_s(x,y),x^{\perp})^{\perp}&=\pi_s(x^{\perp}\vee(x\wedge y),x^{\perp})^{\perp}
        \\&=((x^{\perp}\vee(x\wedge y))^{\perp}\vee((x^{\perp}\vee(x\wedge y))\wedge x^{\perp}))^{\perp}
        \\&=(x^{\perp}\vee(x\wedge y))^{\perp\perp}\wedge((x^{\perp}\vee(x\wedge y))\wedge x^{\perp})^{\perp}
        \\&=(x^{\perp}\vee(x\wedge y))\wedge((x^{\perp}\vee(x\wedge y))\wedge x^{\perp})^{\perp}
        \\&=(x^{\perp}\vee(x\wedge y))\wedge x^{\perp\perp}
        \\&=(x^{\perp}\vee(x\wedge y))\wedge x
        \\&=x\wedge y.
    \end{align*}
   Note that the final equality uses the equational formulation of the orthomodularity law described in Proposition \ref{oml characterization}(2). This completes the proof. 
\end{proof}
\begin{theorem}\label{alg to imp}
    $\langle A;\pi_s,0,(\exists)_{i\in I},(\delta_{i,k})_{i,k}\rangle$ is an $I$-dimensional cylindric quasi-implication algebra whenever $A$ is an $I$-dimensional quantum cylindric algebra. 
\end{theorem}
\begin{proof}
    It follows by Theorem \ref{oml to qia} that the reduct $\langle A;\pi_s,0\rangle$ is a bounded quasi-implication algebra if $A$ is an orthomodular lattice and it follows by McDonald \cite[Theorem 4.6]{mcdonald1} that the reduct $\langle A;\pi_s,0,\exists\rangle$ is a monadic quasi-implication algebra if $\langle A;\exists\rangle$ is a quantum monadic algebra. Condition 2 of Definition \ref{I-dimensional cylindric quasi-implication algebra} is obvious and hence it remains to verify that condition 3 of Definition \ref{I-dimensional cylindric quasi-implication algebra} is satisfied. Assume that $i,l\not=k$. It suffices to demonstrate the following equation: 
   \begin{equation}\label{a}
\exists_k(\pi_s(\pi_s(\pi_s(\delta_{i,k},\delta_{k,l}),\pi_s(\delta_{i,k},0)),0))=\delta_{i,l}
   \end{equation}
 In doing  so, we first note that $\pi_s(x,0)=x^{\perp}\vee(x\wedge 0)=x^{\perp}\vee0=x^{\perp}$ and hence: 
    \begin{equation}\label{b}
\pi_s(\pi_s(\pi_s(\delta_{i,k},\delta_{k,l}),\pi_s(\delta_{i,k},0)),0)=\pi_s(\pi_s(\delta_{i,k},\delta_{k,l}),\delta_{i,k}^{\perp})^{\perp}
    \end{equation}
Therefore by Equation \ref{a} and Equation \ref{b} it suffices to show that: 
    \begin{equation}\label{c}
\exists_k(\pi_s(\pi_s(\delta_{i,k},\delta_{k,l}),\delta_{i,k}^{\perp})^{\perp})=\delta_{i,l}.
    \end{equation}
    The calculation of Equation \ref{c} is then immediate by Lemma \ref{useful lemma} since: 
    \begin{align*}
\exists_k(\pi_s(\pi_s(\delta_{i,k},\delta_{k,l}),\delta_{i,k}^{\perp})^{\perp})&=\exists_k(\delta_{i,k}\wedge\delta_{k,l})=\delta_{i,l}.
    \end{align*}
    This completes the proof. 
\end{proof} 
\begin{theorem}\label{imp to alg}
   Every $I$-dimensional cylindric quasi-implication algebra can be converted into an $I$-dimensional quantum cylindric algebra.  
\end{theorem}
\begin{proof}
    It follows by Theorem \ref{quasi-implication algebra is an orthomodular lattice} that every quasi-implication algebra can be converted into an orthomodular lattice and moreover, it follows by McDonald \cite[Theorem 4.7]{mcdonald2} that $\Diamond\colon A\to A$ is a quantifier for any quasi-implication algebra $A$ and hence every monadic quasi-implication algebra can be converted into a quantum monadic algebra. Conditions 3 and 4(a) of Definition \ref{quantum cylindric algebra} follow immediately from conditions 2 and 3 of Definition \ref{I-dimensional cylindric quasi-implication algebra}. By Lemma \ref{meets} together with condition 4 of Definition \ref{I-dimensional cylindric quasi-implication algebra}, we have:   
      \begin{align*}
        \Diamond_k(\inf\{d_{i,k}, d_{k,l}\})=\Diamond_k(((d_{i,k}\cdot d_{k,l})\cdot(d_{i,k}\cdot 0))\cdot 0)=d_{i,l}
    \end{align*}
    This guarantees that $d_{i,k}$ is a diagonal element for every $i,k\in I$ and hence condition 4(b) of Definition \ref{quantum cylindric algebra} is satisfied. This completes the proof.  
\end{proof}
\begin{definition}\label{alg homo}
    Let $A$ and $A'$ be $I$-dimensional quantum cylindric algebras. A function $h\colon A\to A'$ is a \emph{homomorphism} provided: 
    \par
    \vspace{-.3cm}
    \begin{multicols}{2}
    \begin{enumerate}
        \item $h(x\wedge y)=h(x)\wedge h(y)$; 
        \item $h(x\vee y)=h(x)\vee h(y)$; 
        \item $h(x^{\perp})=h(x)^{\perp}$; 
        \item $h(0)=0'$, $h(1)=1'$;
        \item $h(\exists_ix)=\exists_ih(x)$;
        \item $h(\delta_{i,k})=\delta_{i,k}'$.
    \end{enumerate}
     \end{multicols}
\end{definition}
\begin{definition}\label{imp homo}
    Let $A$ and $A'$ be $I$-dimensional cylindric quasi-implication algebras. A function $h\colon A\to A'$ is a \emph{homomorphism} provided: 
    \begin{enumerate}
        \item $h(x\cdot y)=h(x)\cdot h(y)$; 
        \item $h(\Diamond_ix)=\Diamond_ih(x)$; 
        \item $h(d_{i,k})=d_{i,k}'$;
          \item $h(0)=0'$. 
    \end{enumerate}
\end{definition}
\begin{definition}
    By $\mathbf{QCA}$ we denote the category of $I$-dimensional quantum cylindric algebras and their associated homomorphisms. By $\mathbf{CQIA}$ we denote the category of $I$-dimensional cylindric quasi-implication algebras and their associated homomorphisms. 
\end{definition}
The proceeding lemmas follow immediately from Theorem \ref{alg to imp}, Theorem \ref{imp to alg}, Definition \ref{alg homo}, and Definition \ref{imp homo}. 
\begin{lemma}\label{homo lemma1}
    Let $A$ and $A'$ be I-dimensional quantum cylindric algebras and let $B$ and $B'$ be the corresponding $I$-dimensional cylindric quasi-implication algebras obtained from $A$ and $A'$ respectively via Theorem \ref{alg to imp}. Then $h\colon B\to B'$ is a homomorphism if $h\colon A\to A'$ is a homomorphism.  
\end{lemma}
\begin{lemma}\label{homo lemma2}
    Let $A$ and $A'$ be I-dimensional cylindric quasi-implication algebras and let $B$ and $B'$ be the corresponding $I$-dimensional quantum cylindric algebras obtained from $A$ and $A'$ respectively via Theorem \ref{imp to alg}. Then $h\colon B\to B'$ is a homomorphism if $h\colon A\to A'$ is a homomorphism. 
\end{lemma}
\begin{theorem}
    $\mathbf{QCA}$ is isomorphic to $\mathbf{CQIA}$. 
\end{theorem}
\begin{proof}
Let $A$ be an $I$-dimensional cylindric quasi-implication algebra, let $B$ be the corresponding $I$-dimensional quantum cylindric algebra obtained from $A$ via Theorem \ref{imp to alg}, and let $\widehat{A}$ be the corresponding $I$-dimensional cylindric quasi-implication algebra obtained from $B$ via Theorem \ref{alg to imp}. Since the diagonal $d_{i,k}$ associated with $A$ is clearly identical with the diagonal $\widehat{d_{i,k}}$ associated with $\widehat{A}$ it follows by \cite[Theorem 4.9]{mcdonald2} that $A=\widehat{A}$. Now conversely, let $A$ be an $I$-dimensional quantum cylindric algebra, let $B$ be the corresponding $I$-dimensional cylindric quasi-implication algebra obtained from $A$ via Theorem \ref{alg to imp}, and let $\widehat{A}$ be the corresponding $I$-dimensional quantum cylindric algebra obtained from $B$ via Theorem \ref{imp to alg}. Since the diagonal $\delta_{i,k}$ associated with $A$ is clearly identical with the diagonal $\widehat{\delta_{i,k}}$ associated with $\widehat{A}$ it follows by \cite[Theorem 4.8]{mcdonald2} that $A=\widehat{A}$. Lemma \ref{homo lemma1} together with Lemma \ref{homo lemma2} then guarantees that $\mathbf{QCA}$ is isomorphic to $\mathbf{CQIA}$.  
\end{proof}
\section{Cylindric orthoframes}
In this section, we give two alternative constructions of a cylindric orthoframe from a cylindric quasi-implication algebra.     
\begin{definition}\label{of}
     An \textit{orthoframe} is pair $\langle X,\perp\rangle$ such that $X$ is a set and $\perp\subseteq X^2$ is an orthogonality relation, i.e., $\perp$ is irreflexive and symmetric. Moreover, for any orthoframe $X$ and any subset $U\subseteq X$, define: \[U^{\perp}=\{x\in X: x\perp U\}=\{x\in X:x\perp y\hspace{.2cm}\text{for all}\hspace{.2cm}y\in U\}.\]   
\end{definition}
   The following relational structures were introduced by Harding \cite{harding}. 
\begin{definition}\label{mof}
    A \emph{monadic orthoframe} is a triple $\langle X;\perp,R\rangle$ such that: 
    \begin{enumerate}
        \item $\langle X;\perp\rangle$ is an orthoframe;
        \item $R$ is a binary relation on $X$ that is reflexive and transitive; 
        \item $R[R[\{x\}]^{\perp}]\subseteq R[\{x\}]^{\perp}$ for all $x\in X$.  
    \end{enumerate}
\end{definition}
It is easy to see that the reflexivity of $R$ gives $R[R[\{x\}]^{\perp}]=R[\{x\}]^{\perp}$. 

\begin{definition}\label{col}
An \emph{I-dimensional cylindric orthoframe} is a relational structure $\langle X;\perp,(R_{i})_{i\in I},(\Delta_{i,k})_{i,k\in I}\rangle$ such that for all $i,k,l\in I$:   
    \begin{enumerate}
        \item $\langle X;\perp,R_{i}\rangle$ is a monadic orthoframe; 
        \item $R_i$ commutes with $R_k$;
        \item $\Delta_{i,k}\subseteq X$ with $\Delta_{i,k}=\Delta_{k,i}=\Delta^{\perp\perp}_{i,k}$ and $\Delta_{i,i}=X$;
        \item if $i,l\not=k$, then $R_k[\Delta_{i,k}\cap \Delta_{k,l}]=\Delta_{i,l}$. 
    \end{enumerate}
\end{definition}
Orthoframes, monadic orthoframes, and $I$-dimensional cylindric orthoframes play an important role in the set-theoretic, algebraic, and topological representation theory of ortholattices, monadic ortholattices, and $I$-dimensional cylindric ortholattices, respectively (consult \cite{goldblatt, harding, harding3, mcdonald2} for more details).  

Throughout the remainder of this work, for any cylindric quasi-implication algebra $A$, let $A\setminus\{0\}=\{x\in A:x\not=0\}$ and let $\psi(x)=\{y\in A\setminus\{0\}:y\cdot x=1\}$. 

We now proceed by constructing an $I$-dimensional cylindric orthoframe from the non-zero elements of an $I$-dimensional cylindric quasi-implication algebra. This generalizes the construction given by Harding \cite{harding} in the setting of cylindric ortholattices from the perspective of MacNeille completions.  
\begin{definition}
    Let $A$ be an $I$-dimensional cylindric quasi-implication algebra. The \emph{I-dimensional MacLaren frame} induced by $A$ is a relational structure $X^M_A=\langle A\setminus\{0\};\perp^M_A,(R^M_{i})_{i\in I},(\Delta^M_{i,k})_{i,k\in I}\rangle$ such that: 
    \begin{enumerate}
        \item $\perp^M\subseteq A\setminus\{0\}\times A\setminus\{0\}$ is defined by $x\perp^My\Longleftrightarrow x\cdot (y\cdot0)=1$; 
        \item $R^M_i\subseteq A\setminus\{0\}\times A\setminus\{0\}$ is defined by $xR_i^My\Longleftrightarrow y\cdot\Diamond_i x=1$; 
        \item $\Delta^M_{i,k}\subseteq A\setminus\{0\}$ is defined by $\Delta_{i,k}=\psi(d_{i,k})$. 
    \end{enumerate}
\end{definition}
The following lemmas will be used in the proof of Theorem \ref{algebra to maclaren frame}. 
\begin{lemma}\label{mof lemma1}
    Let $A$ be an I-dimensional cylindric quasi-implication algebra and let $X^M_A$ be the I-dimensional cylindric MacLaren frame induced by $A$. Then: 
    \begin{enumerate}
        \item $R^M_i[\{x\}]^{\perp}=\{y\in A\setminus \{0\}:y\cdot (\Diamond x\cdot 0)=1\}$;
        \item $z\cdot\Diamond_i(\Diamond_i x\cdot 0)=1$ for all $z\in R_i^M[R_i^M[\{x\}]^{\perp}]$.
    \end{enumerate}
\begin{proof}
    The proof of part 1 immediately follows from \cite[Lemma 5.5]{mcdonald2} and the proof of part 2 immediately follows from \cite[Lemma 5.6]{mcdonald2}. 
\end{proof}
\end{lemma}
\begin{lemma}\label{mof lemma2}
    Let $X$ be an $I$-dimensional cylindric quasi-implication algebra and let $X^M_A$ be the $I$-dimensional cylindric MacLaren frame induced by $A$. Then: 
    \begin{enumerate}
        \item $\psi(x)\cap\psi(y)=\psi(((x\cdot y)\cdot(x\cdot 0))\cdot 0)$; 
        \item $\psi(x\cdot 0)=\psi(x)^{\perp}$; 
        \item $\Delta^M_{i,k}=\Delta^M_{k,i}$; 
        \item $\Delta^M_{i,i}=A\setminus\{0\}$. 
    \end{enumerate}
\end{lemma}
\begin{proof}
    For part 1, we first demonstrate the $\psi(((x\cdot y)\cdot(x\cdot 0))\cdot 0)\subseteq\psi(x)\cap\phi(y)$ inclusion. Assume that $z\in \psi(((x\cdot y)\cdot(x\cdot 0))\cdot 0)$ so that $z\cdot (((x\cdot y)\cdot(x\cdot 0))\cdot 0)=1$. By Lemma \ref{partial order} combined with Lemma \ref{meets}, it follows that: 
   \begin{equation}\label{equation1}
       (((x\cdot y)\cdot(x\cdot 0))\cdot0)\cdot x=(((x\cdot y)\cdot(x\cdot 0))\cdot0)\cdot y=1. 
   \end{equation}
   By our hypothesis and Equation \ref{equation1}, we have $z\cdot x=z\cdot y=1$ and hence $z\in\psi(x)$ and $z\in\psi(y)$ so $z\in\psi(x)\cap\psi(y)$. For the $\psi(x)\cap\psi(y)\subseteq\psi(((x\cdot y)\cdot(x\cdot 0))\cdot 0)$ inclusion, assume $z\in\psi(x)\cap\psi(y)$ so that $x\in\psi(x)$ and $z\in\psi(y)$. The definition of $\psi$ then yields $z\cdot x=z\cdot y=1$. Since: 
   \[((x\cdot y)\cdot(x\cdot 0))\cdot0=\max\{w\in A:w\cdot x=1\hspace{.2cm}\text{and}\hspace{.2cm}w\cdot y=1\}\] we have that $z\cdot(((x\cdot y)\cdot(x\cdot 0))\cdot 0)=1$ and hence $z\in\psi(((x\cdot y)\cdot(x\cdot 0))\cdot 0)$. 

   For part 2, assume that $y\in\psi(x\cdot 0)$ so that $y\cdot(x\cdot 0)=1$. Then $y\perp^Mx$ by the definition of $\perp^M$ and hence $y\in\psi(x)^{\perp}$. Now assume that $y\in\psi(x)^{\perp}$. Then $y\cdot (z\cdot 0)=1$ for every $z\in A\setminus\{0\}$ such that $z\cdot x=1$. Clearly $x\cdot x=1$ by Definition \ref{top element} and hence $y\cdot(x\cdot 0)=1$ and thus $y\in\psi(x\cdot 0)$, as desired.

   Conditions 3 and 4 follow immediately by the Definition of $\Delta^M_{i,k}$, $\Delta^M_{k,i}$, Definition \ref{I-dimensional cylindric quasi-implication algebra}(3), and Proposition \ref{top element lemma} since $\Delta^M_{i,k}=\psi(d_{i,k})=\psi(d_{k,i})=\Delta^M_{k,i}$ and: \[\Delta^M_{i,i}=\psi(d_{i,i})=\psi(1)=\{x\in A\setminus\{0\}:x\cdot 1=1\}=A\setminus\{0\}.\] This completes the proof.   \end{proof}
\begin{theorem}\label{algebra to maclaren frame}
    If $A$ is an $I$-dimensional cylindric quasi-implication algebra, then its MacLaren frame $X^M_A$ is an $I$-dimensional cylindric orthoframe. 
\end{theorem}
\begin{proof}
    It follows by McDonald \cite[Theorem 5.7]{mcdonald1} by way of Lemma \ref{mof lemma1} that the reduct $\langle\mathfrak{F}(A);\perp^M,R^M\rangle$ of $X^M_A$ is a monadic orthoframe whenever $A$ is a monadic quasi-implication algebra and hence we first verify that $R^M_i\circ R^M_k=R^M_k\circ R^M_i$ for all $i,k\in I$. Assume that $xR^M_iy$ and $yR^M_kz$. The definition of $R^M_i$ and $R^M_k$ gives $y\cdot\Diamond_ix=z\cdot\Diamond_ky=1$ so in particular, we have $z\cdot\Diamond_k\Diamond_ix=z\cdot\Diamond_i\Diamond_kx=1$. If we set $w=\Diamond_kx$, then $xR^M_kw$ and $wR^M_iz$ so $R^M_i\circ R^M_k\subseteq R^M_k\circ R^M_i$. The converse inclusion follows by a symmetric argument. 

    For condition 3 of Definition \ref{col}, notice that by Lemma \ref{mof lemma2}(2), we have: \[\Delta^{M\perp\perp}_{i,k}=\psi(d_{i,k})=\psi((d_{i,k}\cdot 0)\cdot 0)=\psi(d_{i,k}\cdot 0)^{\perp}=\psi(d_{i,k})^{\perp\perp}=\Delta_{i,k}^{M\perp\perp}.\] The remainder of condition 3 follows by Lemma \ref{mof lemma2}(3) and Lemma \ref{mof lemma2}(4).

    For condition 4, it remains to verify that $R^M_k[\Delta^M_{i,k}\cap\Delta^M_{k,l}]=\Delta_{i,l}$ whenever $i,l\not=k$. By the definition of $\Delta^M_{i,k}$ and $\Delta^M_{k,l}$:
   \begin{equation}\label{1}
       R^M_{k}[\Delta^M_{i,k}\cap \Delta^M_{k,l}]=R^M_{k}[\psi(\delta^M_{i,k})\cap\psi(d_{k,l})]
   \end{equation}
   By condition 1 of Lemma \ref{mof lemma2} we have:
   \begin{equation}\label{2}
\psi(d_{i,k})\cap\psi(d_{k,l})=\psi(((d_{i,k}\cdot d_{k,l})\cdot(d_{i,k}\cdot 0))\cdot0)
   \end{equation}
 Therefore, by Equation \ref{1} and Equation \ref{2} it suffices to demonstrate: 
   \begin{equation}
       R^M_k[\psi(((d_{i,k}\cdot d_{k,l})\cdot(d_{i,k}\cdot 0))\cdot0)]=\psi(d_{i,l}).
   \end{equation}  
    For the $R^M_k[\psi(((d_{i,k}\cdot d_{k,l})\cdot(d_{i,k}\cdot 0))\cdot0)]\subseteq\psi(d_{i,l})$ inclusion, assume that $y\in R^M_k[\psi(((d_{i,k}\cdot d_{k,l})\cdot(d_{i,k}\cdot 0))\cdot0)]$. Then $xR_k^My$ for some $x\in \psi(((d_{i,k}\cdot d_{k,l})\cdot(d_{i,k}\cdot 0))\cdot0)$ so that $y\cdot \Diamond_kx=1$ by the definition of $R^M_k$. Moreover, we have:
\begin{equation}\label{monoequation}
    x\cdot(((d_{i,k}\cdot d_{k,l})\cdot(d_{i,k}\cdot 0))\cdot0)=1
\end{equation}
by the definition of $\psi$. The monotonicity of $\Diamond_k$ applied to Equation \ref{monoequation} gives: 
\begin{equation}\label{d}
    \Diamond_kx\cdot\Diamond_k(((d_{i,k}\cdot d_{k,l})\cdot(d_{i,k}\cdot 0))\cdot0)=1
\end{equation}
Our hypothesis along with Equation \ref{d} and Definition \ref{I-dimensional cylindric quasi-implication algebra}(4) then yields:
\[y\cdot\Diamond_k(((d_{i,k}\cdot d_{k,l})\cdot(d_{i,k}\cdot 0))\cdot0)=y\cdot d_{i,l}=1\] and hence $y\in\psi(d_{i,l})$, as desired. For the $\psi(d_{i,l})\subseteq R^M_k[\psi(((d_{i,k}\cdot d_{k,l})\cdot(d_{i,k}\cdot 0))\cdot0)]$ inclusion, assume that $y\in\psi(d_{i,l})$. By Definition \ref{I-dimensional cylindric quasi-implication algebra}(4), we have that $y\in\psi(\Diamond_k(((d_{i,k}\cdot d_{k,l})\cdot(d_{i,k}\cdot 0))\cdot0))$. The definition of $\psi$ then gives:
\begin{equation}\label{e}
    y\cdot \Diamond_k(((d_{i,k}\cdot d_{k,l})\cdot(d_{i,k}\cdot 0))\cdot0)=1
\end{equation}
 It suffices to find some $x\in \psi(((d_{i,k}\cdot d_{k,l})\cdot(d_{i,k}\cdot 0))\cdot0)$ such that $xR^M_ky$. Hence, set $x=((d_{i,k}\cdot d_{k,l})\cdot(d_{i,k}\cdot 0))\cdot0$. Equation \ref{e} together with the definition of $R^M_k$ yields $xR^M_ky$ and hence $y\in R^M_k[\psi(((d_{i,k}\cdot d_{k,l})\cdot(d_{i,k}\cdot 0))\cdot0)]$.  
\end{proof}

We now proceed by constructing an $I$-dimensional cylindric orthoframe from the proper filters of an $I$-dimensional cylindric quasi-implication algebra. This generalizes the construction given by McDonald \cite{mcdonald2} in the setting of cylindric ortholattices from the perspective of canonical completions. 

\begin{definition}
Let $A$ be a bounded lattice. Then:
    \begin{enumerate}
    \item a non-empty subset $x\subseteq A$ is a \emph{filter} whenever:
    \begin{enumerate}
        \item $a\in x$ and $a\leq b$ implies $b\in x$;
        \item $a\in x$ and $b\in x$ implies $a\wedge b\in x$; 
    \end{enumerate}
    \item a filter $x\subseteq A$ is \emph{proper} whenever $x\not=A$. Note that this is equivalent to the requirement that $0\not\in x$ since $0\leq a$ for all $a\in A$. 
    \end{enumerate}
\end{definition} 

    The definition of the partial-order structure $\preceq$ induced by a bounded quasi-implication algebra $A$, the definition of meets in $A$, and Lemma \ref{meets} suggests the following definition of a proper filter in a bounded quasi-implication algebra. 
\begin{definition}\label{filter}
    Let $A$ be a bounded quasi-implication algebra. A \emph{filter} on $A$ is a non-empty subset $\alpha\subseteq A$ satisfying the following conditions: 
    \begin{enumerate}
        \item if $x\in\alpha$ and $x\cdot y=1$, then $y\in\alpha$; 
        \item if $x\in \alpha$ and $y\in \alpha$, then $((x\cdot y)\cdot(x\cdot 0))\cdot 0\in \alpha$
    \end{enumerate}
    We regard a filter $\alpha\subseteq A$ as \emph{proper} whenever $0\not\in \alpha$.  
\end{definition}

Throughout the remainder of this work, if $A$ is an $I$-dimensional cylindric quasi-implication algebra, let $\mathfrak{F}(A)$ denote the collection of all proper filters of $A$, let $\Diamond_i[\alpha]=\{\Diamond_ix:x\in\alpha\}$, and for any $x\in A$, let $\phi(x)=\{\alpha\in\mathfrak{F}(A):x\in\alpha\}$.  

\begin{definition}
    Let $A$ be an $I$-dimensional cylindric quasi-implication algebra. The \emph{I-dimensional Goldblatt frame} induced by $A$ is a relational structure $X^G_A=\langle \mathfrak{F}(A);\perp^G_A,(R^G_{i})_{i\in I},(\Delta^G_{i,k})_{i,k\in I}\rangle$ such that: 
    \begin{enumerate}
        \item $\perp^G\subseteq\mathfrak{F}(A)\times\mathfrak{F}(A)$ is defined by $\alpha\perp^G\beta$ iff there exists some $x\in A$ such that $x\in\alpha$ and $x\cdot0\in\beta$; 
        \item $R^G_i\subseteq\mathfrak{F}(A)\times\mathfrak{F}(A)$ is defined by $\alpha R^G\beta\Longleftrightarrow\Diamond_i[\alpha]\subseteq\beta$; 
        \item $\Delta^G_{i,k}\subseteq\mathfrak{F}(A)$ is defined by $\Delta^G_{i,k}=\phi(d_{i,k})$.
    \end{enumerate}
\end{definition}
The following lemma will be used to simplify the proof of Theorem \ref{algebra to goldblatt frame}. 
\begin{lemma}\label{lemma 5.10}
    Let $A$ be an I-dimensional cylindric quasi-implication algebra and let $X^G_A$ be the $I$-dimensional cylindric Goldblatt frame induced by $A$. Then:
    \begin{enumerate}
        \item $\phi(x)\cap\phi(y)=\phi(((x\cdot y)\cdot(x\cdot 0))\cdot0)$;
        \item $\phi(x\cdot 0)=\phi(x)^{\perp}$; \item $\Delta^G_{i,k}=\Delta^G_{k,i}$; 
        \item $\Delta^G_{i,i}=\mathfrak{F}(A)$. 
    \end{enumerate}
\end{lemma}
\begin{proof}
For part 1, we first consider the $\phi(((x\cdot y)\cdot(x\cdot 0))\cdot0)\subseteq\phi(x)\cap\phi(y)$ inclusion. Assume $\alpha\in\phi(((x\cdot y)\cdot(x\cdot 0))\cdot0)$. By the definition of $\phi$, we have $((x\cdot y)\cdot(x\cdot0))\cdot0)\in\alpha$. By Lemma \ref{partial order} combined with Lemma \ref{meets}: 
   \begin{equation}\label{12}
       (((x\cdot y)\cdot(x\cdot 0))\cdot0)\cdot x=(((x\cdot y)\cdot(x\cdot 0))\cdot0)\cdot y=1. 
   \end{equation}
   Since $\alpha$ is a filter and $((x\cdot y)\cdot(x\cdot0))\cdot0)\in\alpha$, Equation \ref{12} together with Definition \ref{filter}(1) implies $x\in\alpha$ and $y\in\alpha$ so $\alpha\in\phi(x)$ and $\alpha\in\phi(y)$, whence $\alpha\in\phi(x)\cap\phi(y)$. The $\phi(x)\cap\phi(y)\subseteq\phi(((x\cdot y)\cdot(x\cdot 0))\cdot0)$ inclusion runs analogously.  

For part 2, let $\alpha\in\phi(x\cdot 0)$ so that $x\cdot0\in\alpha$. Then for any $\beta\in\phi(x)$, we have $x\in\beta$ and hence $\alpha\perp^G\beta$. Hence $\alpha\perp^G\phi(x)$ which implies $\alpha\in\phi(x)^{\perp}$. Conversely, assume $\alpha\in\phi(x)^{\perp}$ and let $\beta$ be the principal filter generated by $x$ so that $\beta={\uparrow}x=\{y\in A:x\cdot y=1\}$. Clearly $\beta\in\phi(x)$ and thus $\alpha\perp^G\beta$ which implies that $y\cdot 0\in\alpha$ and $y\in\beta$ for some $y\in A$. Our construction of $\beta$ implies that $x\cdot y=1$. Clearly $0\cdot x=1$ for all $x\in A$ by Proposition \ref{top element lemma} so $0\cdot((y\cdot x)\cdot 0)=1$. By Proposition \ref{left monotone}, we have $(y\cdot 0)\cdot(y\cdot((y\cdot x)\cdot 0))=1$ but: 
\begin{align*}
    y\cdot((y\cdot x)\cdot0)&=(y\cdot x)\cdot(y\cdot 0)\tag{Lemma \ref{lemma}(4)}
    \\&=(x\cdot y)\cdot(x\cdot0)\tag{Definition \ref{quasi-implication algebra}(2)}
    \\&=1\cdot(x\cdot0)\tag{since $x\cdot y=1$}
    \\&=x\cdot0\tag{Proposition \ref{top element lemma}}
\end{align*}
Therefore we have that $x\cdot 0\in\alpha$ and hence $\alpha\in\phi(x\cdot 0)$. For part 3, applying the definition of $\Delta^G_{i,k}$, $\Delta^G_{k,i}$, $\phi$, and Definition \ref{I-dimensional cylindric quasi-implication algebra}(2) yields the following: \[\Delta^G_{i,k}=\phi(d_{i,k})=\{\alpha\in\mathfrak{F}(A):d_{i,k}\in\alpha\}=\{\alpha\in\mathfrak{F}(A):d_{k,i}\in\alpha\}=\phi(d_{k,i})=\Delta^G_{k,i}.\]
   For part 4, we note that $\Delta^G_{i,i}=\phi(d_{i,i})=\phi(1)=\{\alpha\in\mathfrak{F}(A):1\in\alpha\}=\mathfrak{F}(A)$ since $x\cdot 1=1$ for all $x\in A$ by Proposition \ref{top element lemma} and $\alpha$ is an upset.  
\end{proof}
\begin{theorem}\label{algebra to goldblatt frame}
    If $A$ is an $I$-dimensional cylindric quasi-implication algebra, then its Goldblatt frame $X^G_A$ is an $I$-dimensional cylindric orthoframe. 
\end{theorem}
\begin{proof}
    It follows by McDonald \cite[Theorem 5.9]{mcdonald1} that the reduct $\langle\mathfrak{F}(A);\perp^G,R^G\rangle$ of $X^G_A$ is a monadic orthoframe whenever $A$ is a monadic quasi-implication algebra and hence we first verify that $R^G_i\circ R^G_k=R^G_k\circ R^G_i$ for all $i,k\in I$. Assume that $\alpha R^G_i\beta$ and $\beta R^G_k\gamma$. By the definition of $R^G_i$ and $R^G_k$, this implies $\Diamond_i[\alpha]\subseteq\beta$ and $\Diamond_k[\beta]\subseteq\gamma$. It suffices to find some proper filter $\epsilon\in\mathfrak{F}(A)$ such that $\alpha R^G_k\epsilon$ and $\epsilon R^G_i\gamma$. Therefore, let: 
    \[\epsilon=\bigcap\{\eta\in\mathfrak{F}(A):\Diamond_k[\alpha]\subseteq\eta\}.\] Observe that $\epsilon$ is in particular, the proper filter generated by $\Diamond_k[\alpha]$ and that $\epsilon$ is proper since $\alpha$ is assumed to be proper. Given our construction of $\epsilon$, it is obvious that $\Diamond_k[\alpha]\subseteq\epsilon$ and $\Diamond_i[\epsilon]\subseteq\gamma$ so that $\alpha R_k\epsilon$ and $\epsilon R_i\gamma$. This implies $R^G_i\circ R^G_k\subseteq R^G_k\circ R^G_i$. A symmetric argument exhibits the inclusion $R^G_k\circ R^G_i=R^G_i\circ R^G_k$ and hence we conclude that $R^G_i\circ R^G_k=R^G_k\circ R^G_i$, as desired.

    To verify that $\Delta^G_{i,k}=\Delta_{i,k}^{G\perp\perp}$, observe that by Lemma \ref{lemma 5.10}(2), we have: 
    \[\Delta^G_{i,k}=\phi(d_{i,k})=\phi((d_{i,k}\cdot0)\cdot0)=\phi(d_{i,k}\cdot0)^{\perp}=\phi(d_{i,k})^{\perp\perp}=\Delta_{i,k}^{G\perp\perp}\] The remainder of condition 3 follows from conditions 3 and 4 of Lemma \ref{lemma 5.10}. It remains to verify $R^G_k[\Delta^G_{i,k}\cap\Delta^G_{k,l}]=\Delta^G_{i,l}$. By the definition of $\Delta^G_{i,k}$ and $\Delta^G_{k,l}$:
   \begin{equation}\label{16}
       R^G_{k}[\Delta^G_{i,k}\cap \Delta^G_{k,l}]=R^G_{k}[\phi(d_{i,k})\cap\phi(d_{k,l})]
   \end{equation}
   By condition 1 of Lemma \ref{lemma 5.10} we have:
   \begin{equation}\label{15}
\phi(d_{i,k})\cap\phi(d_{k,l})=\phi(((d_{i,k}\cdot d_{k,l})\cdot(d_{i,k}\cdot 0))\cdot0)
   \end{equation}
 Therefore, by Equation \ref{16} and Equation \ref{15} it suffices to demonstrate: 
   \begin{equation}\label{17}
       R^G_k[\phi(((d_{i,k}\cdot d_{k,l})\cdot(d_{i,k}\cdot 0))\cdot0)]=\phi(d_{i,l}).
   \end{equation}
   Assume that $\beta\in R^G_k[\phi(((d_{i,k}\cdot d_{k,l})\cdot(d_{i,k}\cdot 0))\cdot0)]$. Then $\alpha R^G_k\beta$ for some $\alpha\in\phi(((d_{i,k}\cdot d_{k,l})\cdot(d_{i,k}\cdot 0))\cdot0)$ and hence $\Diamond_k[\alpha]\subseteq\beta$. Then $((d_{i,k}\cdot d_{k,l})\cdot(d_{i,k}\cdot 0))\cdot0\in\alpha$ so $\Diamond_k(((d_{i,k}\cdot d_{k,l})\cdot(d_{i,k}\cdot 0))\in\beta$. By Definition \ref{I-dimensional cylindric quasi-implication algebra}(3), we have: 
   \begin{equation}\label{32}
       \Diamond_k((d_{i,k}\cdot d_{k,l})\cdot(d_{i,k}\cdot 0))\cdot0=d_{i,l}
   \end{equation}
   and therefore $d_{i,l}\in\beta$ so $\beta\in\phi(d_{i,l})$ which establishes the left-to-right inclusion of Equation \ref{17}. For the converse inclusion, take $\beta\in\phi(d_{i,l})$. Equation \ref{32} implies $\beta\in\phi(\Diamond_k((d_{i,k}\cdot d_{k,l})\cdot(d_{i,k}\cdot 0))\cdot0)$ and hence $\Diamond_k((d_{i,k}\cdot d_{k,l})\cdot(d_{i,k}\cdot 0))\cdot0\in\beta$. We want to find some proper filter $\alpha$ satisfying the conditions that $\alpha\in\phi(((d_{i,k}\cdot d_{k,l})\cdot(d_{i,k}\cdot 0))\cdot0)$ and $\Diamond_k[\alpha]\subseteq\beta$. We construct such an $\alpha$ by:
   \begin{equation}
       \alpha={\uparrow}(((d_{i,k}\cdot d_{k,l})\cdot(d_{i,k}\cdot 0))\cdot0)=\{x\in A:(((d_{i,k}\cdot d_{k,l})\cdot(d_{i,k}\cdot 0))\cdot0)\cdot x=1\}
   \end{equation}
   It is an easy exercise to see that $\alpha$ is a proper filter. In particular, $\alpha$ is the principal filter generated by $((d_{i,k}\cdot d_{k,l})\cdot(d_{i,k}\cdot 0))\cdot0$. Note that by Definition \ref{top element}, we have: 
   \begin{equation}
       (((d_{i,k}\cdot d_{k,l})\cdot(d_{i,k}\cdot 0))\cdot0)\cdot(((d_{i,k}\cdot d_{k,l})\cdot(d_{i,k}\cdot 0))\cdot0)=1
   \end{equation}
   and hence $((d_{i,k}\cdot d_{k,l})\cdot(d_{i,k}\cdot 0))\cdot0\in\alpha$ by our construction of $\alpha$. Therefore we have that $\alpha\in\phi(((d_{i,k}\cdot d_{k,l})\cdot(d_{i,k}\cdot 0))\cdot0)$. Hence choose any $x\in\Diamond_k[\alpha]$ so that $x=\Diamond_ky$ for some $y\in\alpha$. We consider two cases. First, if $y=((d_{i,k}\cdot d_{k,l})\cdot(d_{i,k}\cdot 0))\cdot0$, then $\Diamond_ky=\Diamond_k(((d_{i,k}\cdot d_{k,l})\cdot(d_{i,k}\cdot 0))\cdot0)$ which implies $\Diamond_ky=x\in\beta$, as desired. If $(((d_{i,k}\cdot d_{k,l})\cdot(d_{i,k}\cdot 0))\cdot0)\cdot y=1$ with $(((d_{i,k}\cdot d_{k,l})\cdot(d_{i,k}\cdot 0))\cdot0)\not=y$, then by Proposition \ref{idempotent}, we have: 
   \begin{equation}\label{13}
       \Diamond_k(((d_{i,k}\cdot d_{k,l})\cdot(d_{i,k}\cdot 0))\cdot0)\cdot\Diamond_ky=1.
   \end{equation}
 Since $\Diamond_k(((d_{i,k}\cdot d_{k,l})\cdot(d_{i,k}\cdot 0))\cdot0)\in\beta$, Equation \ref{13} together with Definition \ref{filter}(1) implies $\Diamond_ky=x\in\beta$, as desired. Therefore we conclude that $\phi(d_{i,l})\subseteq R^G_k[\phi(((d_{i,k}\cdot d_{k,l})\cdot(d_{i,k}\cdot 0))\cdot0)]$, and hence we have established Equation \ref{17}, which completes the proof.   
\end{proof}

\section{Conclusions}
In this note, we studied the operation of Sasaki hook in quantum cylindric algebras by introducing cylindric quasi-implication algebras. We showed that the category $\mathbf{QCA}$ of quantum cylindric algebras is isomorphic to the category $\mathbf{CQIA}$ of cylindric quasi-implication algebras. We then generalized the constructions of cylindric orthoframes in \cite{harding, mcdonald2} to the setting of cylindric quasi-implication algebras. It is worth making an analogous observation that was pointed out in \cite[Section 6]{mcdonald1} that the construction of the corresponding cylindric orthoframe induced by a cylindric quasi-implication algebra $A$ does not make use of the orthomodularity of the cylindric ortholattice induced by $A$.   
\section*{Declarations}

\paragraph{Ethical approval}
Not applicable. 

\paragraph{Availability of data and materials}
Not applicable. 

\paragraph{Competing interests} The author declares no competing interests. 

\par

\end{document}